\documentclass[12pt]{amsart}
\usepackage{mathrsfs}
\usepackage{amsmath}
\usepackage{amsfonts}
\usepackage{amssymb}
\usepackage[hidelinks]{hyperref}
\usepackage{subfigure}
\usepackage{color}
\usepackage{fancyhdr}
\usepackage{bbm}
\usepackage{extarrows}
\usepackage{bm}
\usepackage{graphicx}
\usepackage{cite}
\usepackage[nameinlink]{cleveref} 
\newtheorem{theorem}{Theorem}[section]
\newtheorem{remark}{Remark}[section]
\newtheorem{corollary}{Corollary}[section]
\newtheorem{lemma}{Lemma}[section]
\numberwithin{equation}{section}

\usepackage{geometry}
\begin{document}

\title[mean square of the error term for number of lattice points in a 2-D area]{On the mean square of the error term for the number of lattice points in a two-dimensional area}

\thanks{This work is partially supported by the National Natural Science Foundations of China(Grant	Nos. 12471009, 12301006) and partially supported by Beijing Natural Science Foundation (Grant No. 1242003).}

\author{Lirui Jia}
\address{School of Mathematics and Statistics, Beijing Jiaotong University\\
Beijing 100044, People's Republic of China }

\email{jialirui@126.com}

\author{Wenguang Zhai}
\address{Department of Mathematics, China University of Mining and Thechnology\\ Beijing 100083, People's Republic of China }

\email{zhaiwg@hotmail.com}

\date{}{}


\begin{abstract}
	Suppose $a,~b$ are fixed algebraic numbers with $1\leq a<b$. Let $\Delta_{a,b}(x)$ be  the error term for the number of lattice points in a two-dimensional area $h^ar^b\leq x $ with $h, r$ positive integers.  In this paper, we establish an asymptotic formula for the mean square of $\Delta_{a,b}(x)$ when $a, b$ are fixed algebraic numbers such that $\dfrac{a}{b}$ is irrational, and improve the error term in the previous asymptotic formula for $a, b$ integers with $(a, b)=1$. Based on these asymptotic formulas, we derive estimates for the sign changes of $\Delta_{a,b}(x)$.
\end{abstract}

\maketitle


\section{Introduction and main results}


 Let $a,~b$ be two fixed positive real numbers such that $1\leq a\leq b$. Define $$D_{a,b}(x)=\#\{(h,r):h,r\in \mathbb{Z}^+; h^ar^b\leq x\},$$ which is the number of lattice points in the first quadrant under the curve $ h^ar^b= x$. 
 
 If $a$, $b$ are integers, define 
$$d_{a,b}(n)=\#\{(h,r):h,r\in \mathbb{Z}^+;  h^ar^b=n\},$$ as the general two dimensional function, then $$D_{a,b}(x)=\sum_{n\leq x}d_{a,b}(n).$$

In case of $a=b$, the estimate of $D_{a,b}(x)$ is essentially the same with $D_{1,1}(x)$, so we put $a=b=1$ without loss of generality. 

It can be deduced that  
$$D_{a,b}(x):=\left\{\begin{array}{ll}
	\zeta (\frac{b}{a})x^{\frac{1}{a}}+\zeta (\frac{a}{b})x^{\frac{1}{b}}+\Delta_{a,b}(x),&\text{if }~ a\neq b;\\
	x\log x+(2\gamma-1)x+\Delta_{1,1}(x),&\text{if }~  a=b=1; \end{array} \right.$$
where $\gamma$ is the Euler constant. So the  problem is to study the error term $\Delta_{a,b}(x)$.

When $a=b=1$, $\Delta_{1,1}(x)$ is the error term of the well-known Dirichlet divisor problem. Dirichlet first proved that $\Delta(x)\ll x^{\frac{1}{2}}$. Throughout the past more than 170 years, there are many improvements on this estimate.
The best estimate to-date is given by
Huxley\cite{huxley03Exponential}, reads
\begin{equation*}
	\Delta_{1,1}(x)\ll x^{\frac{131}{416}+\varepsilon}.
\end{equation*}

When $a, b$ are real numbers with $a\neq b$, Richert\cite{richert1952Uber} proved that
\begin{equation}\label{upbd}
	\Delta_{a,b}(x)\ll \left\{\begin{array}{ll} x^{\frac{2}{3(a+b)}},&\text{if }~b<2a;\\x^{\frac{2}{9a}}\log x,&\text{if }~b=2a;\\ x^{\frac{2}{5a+2b}},&\text{if }~b>2a.\end{array} \right. 
\end{equation}
Later, \cite{rankin1955Van,schmidt1964Abschatzungena} improved the estimate for the case $a=1,b=2$, and \cite{kratzel1988Lattice} provided better estimates for some special cases.

It is conjectured that
\begin{equation*}
\Delta_{a,b}(x) \ll x^{\frac{1}{2(a+b)}+\varepsilon}
\end{equation*} 
 is admissible and the best possible for any $1\leq a\leq b$. 
 
 This conjecture is supported by the estimates of power moments of $\Delta_{a,b}(x)$. For the mean square of $\Delta_{1,1}(x)$, Cram{\'e}r \cite{cramer1922zwei} first proved 
 \[\int_1^T\Delta_{1,1}^2(x)dx=\frac{(\zeta(3/2))^4}{6\pi^2\zeta(3)}T^{\frac{3}{2}}+O(T^{\frac{5}{4}+\varepsilon}),\quad\forall~\varepsilon>0. \]
 The estimate $O(T^{\frac{5}{4}+\varepsilon})$ here was improved to $O(T\log^5T)$ by \cite{Dong1956ChuShuWen}, to $O(T\log^4T)$ by \cite{preissmann1988moyenne}, and to $O(T\log^3T\log\log T)$ by \cite{lau2009mean}. The mean value and higher-power moments of $\Delta_{1,1}(x)$ were studied in  \cite{voronoi1904fonction,heath1992distribution,tsang1992higher,ivic1983large,zhai2004higher,ivic2007higher,zhai2004higherpowerb,zhai2005higherpowera}. 
 
 When $a\neq b$ are positive integers, by using the theory of the Riemann zeta-function, Ivi{\'c} \cite{ivicGeneralDivisorProblem1987} proved that \[\int_1^T\Delta_{a,b}^2(x)dx\left\{\begin{array}{ll} \ll T^{1+\frac{1}{a+b}+\varepsilon},\\=\Omega( T^{1+\frac{1}{a+b}}\log^BT),\end{array}\right. \]
where $B>0$ is a constant.  Zhai and Cao \cite{zhai2010mean} obtained that
\begin{equation}
	\int_1^T\Delta_{a,b}^2(x)dx=c_{a,b}T^{1+\frac{1}{a+b}}+O(T^{1+\frac{1}{a+b}-\frac{a}{2b(a+b)(a+b-1)}}\log^{\frac{7}{2}}T),\label{zhai2010}
\end{equation}
where  $c_{a,b}$ is a  positive constant. The error term was improved to $O(T^{1+\frac{1}{a+b}-\frac{a}{b(a+b)(a+b-1)}+\varepsilon})$ in\cite{cao2016Mean}.

When $a=b=1$, most power moment results of $\Delta_{1,1}$ were proved on the basis of the well-known truncated Voronoi's formula (see e.g.\cite[(3.17)]{ivic2003Riemann})
\begin{equation*}
	\Delta_{1,1}(x)=\frac{x^{\frac{1}{4}}}{\sqrt{2}\pi}\sum_{n\leq N}\frac{d(n)}{n^{\frac{3}{4}}}\cos (4\pi\sqrt{nx}-\frac{\pi}{4})+O(x^{\varepsilon}+x^{\frac{1}{2}+\varepsilon}N^{-\frac{1}{2}}),
\end{equation*}
where $1\ll N\ll x^A$ for some fixed $A>0$.

When $a\neq b$ are fixed positive integers, Zhai and Cao \cite{zhai2010mean} got a Voronoi type formula for $\Delta_{a,b}(x)$, by which they proved the mean square asymptotic formula \eqref{zhai2010}.

When $a, b$ are real numbers , Krätzel \cite{kratzel1969Teilerproblem} proved an asymptotic formula for the mean value of $\Delta_{a,b}(x)$ as
\begin{equation}\label{meanv}
	 \int _ {1}^{T} \Delta_{a,b}(x)dx = \frac {T} {4}+O \bigl (T^ {1-\frac{1}{2 (a+ b)}}\bigr).
\end{equation}

In this paper, we will study the mean square of $\Delta_{a,b}(x)$ when $a\neq b$ are algebraic numbers. To state the results, we first denote
\begin{equation*}
	G_{a,b}=\mathop{\sum_{h_1=1}^{\infty}\sum_{h_2=1}^{\infty}\sum_{j_1=1}^{\infty}\sum_{j_2=1}^{\infty}}_{h_1^ar_1^b=h_2^ar_2^b}(h_1h_2)^{-\frac{a+2b}{2(a+b)}}(r_1r_2)^{-\frac{2a+b}{2(a+b)}},
\end{equation*}
which is  a convergent  infinity series  (See Remark \ref{lem:gab} in section \ref{sr&sm}).


Our main results are as follows.
\begin{theorem}\label{thm:meansquare_i}
	Let $1\leq a<b$ be fixed algebraic numbers such that $\dfrac{a}{b}$ is irrational. Suppose $T>10$ is a large parameter, $0<T_0\leq T$ is a real number.  Then there exists  a positive constant $c_4(a,b)$, such that
  	 \begin{multline*}
  		\int_{T}^{T+T_0}\Delta^{2}_{a,b}(x)dx=\frac{a^\frac{b}{a+b}b^\frac{a}{a+b}}{2\pi^2(a+b)}G_{a,b}\int_{T}^{T+T_0}x^\frac{1}{a+b}dx\\
  		+O\Big(T^{1+\frac{1}{a+b}}\exp\big\{-c_4(a,b)(\log T)^\frac{1}{2}(\log\log T)^{-\frac{1}{2}}\big\}\Big).
  	\end{multline*}
\end{theorem}

If \(\frac{a}{b}\) is rational, then there exist positive integers \(a_0\) and \(b_0\) with \((a_0, b_0) = 1\) such that \(\frac{a}{b} = \frac{a_0}{b_0}\), and estimates for \(\Delta_{a,b}(x)\) follow directly from those of \(\Delta_{a_0,b_0}(x)\). So we  assume, without loss of generality, that $a, b$ are positive integers  and have the following asymptotic formula.
\begin{theorem}\label{thm:meansquare_q}
	Let $1\leq a<b$ be fixed integers with $(a, b)=1$.  Suppose $T>10$ is a large parameter, $0<T_0\leq T$ is a real number. Then
	\begin{equation*}
		\int_{T}^{T+T_0}\Delta^{2}_{a,b}(x)dx=\frac{a^\frac{b}{a+b}b^\frac{a}{a+b}}{2\pi^2(a+b)}G_{a,b}\int_{T}^{T+T_0}x^\frac{1}{a+b}dx
		+O\Big(T^{1+\frac{1}{a+b}-\frac{a}{b(a+b)(a+b-1)}}\log^4T\Big).
	\end{equation*}
\end{theorem}
\begin{remark}
	This improves the error term given in \cite{cao2016Mean}.
\end{remark}

From \ref{thm:meansquare_i} and Theorem \ref{thm:meansquare_q}, we derive the following corollaries.

\begin{corollary}\label{sign_i}
	Let $1\leq a<b$, $T>10$, $c_4(a,b)> 0$  be as Theorem \ref{thm:meansquare_i}. There exists a sufficiently large number \( c_5(a, b) > 0 \), such that the function \( \Delta_{a,b}(x) \) changes sign at least once in the interval \(\big[T, T + c_5(a, b)T \exp\big\{-c_4(a, b)(\log T)^{\frac{1}{2}}(\log \log T)^{-\frac{1}{2}}\big\}\big]\).
\end{corollary}

\begin{corollary}\label{sign_q}
	Let $1\leq a<b$, $T>10$ be as Theorem \ref{thm:meansquare_q}. There exists a sufficiently large number \( c_6(a, b) > 0 \), such that the function $\Delta_{a,b}(x)$ changes sign at least once in the interval $\big[T, T + c_6(a,b)T^{1-\frac{a}{b(a+b)(a+b-1)}}\log^4 T\big]$.
\end{corollary}

\textbf{Notations} For a real number $t$, let $[t]$ be the largest integer no greater than $t$, $\{t\}=t-[t]$, $\psi(t)=\{t\}-\frac{1}{2}$, $\|t\|=\min(\{t\}$, $1-\{t\})$, $e(t)=e^{2\pi it}$; $\mathbb{R}$, $\mathbb{Q}$, $\mathbb{\overline{Q}}$ , $\mathbb{Z}$,  $\mathbb{Z}^+$  denote the set of real numbers, of rational numbers, of algebraic numbers, of integers and of positive integers, respectively;  $f\asymp g$ means that both $f\ll g$ and $f\gg g$ hold; $n\sim N$ means $N<n\leq 2N$;  and $\mathop{{\sum}'}\limits_{n\leq x}f(n)$ indicates that if $x$  is an integer, only $\frac{1}{2}f(x)$ is counted. Finally, let  \(\varepsilon > 0\) be a sufficiently small constant, and \(\mathcal{L} = \log T\) throughout this paper.

\section{Preliminaries}\label{yubei}

To prove the theorems, we need the following Lemmas. Lemma \ref{lem:1} is well-known, see, for example, Heath-Brown \cite{heath-brown1983PjateckiiSapiro}. Lemma \ref{lem:liu} is Theorem 1 of Liu \cite{liu1999fundamentala}. Lemma \ref{lem:AgDp} is the main theorem of \cite{roth1955rational}.
\begin{lemma}\label{lem:1}
	Let $H\geq2$ be a real number. Then
	\begin{equation*}
		\psi(u)=-\sum_{1\leq |h|\leq H}\frac{e(hu)}{2\pi ih}+O\Big(\min\big(1,\frac{1}{H\| u\|\vert}\big)\Big).
	\end{equation*}
\end{lemma}

\begin{lemma}\label{lem:liu}
	Let $f(x)$ be a real function such that $f^{(5)}(x)$ is a continuous function for $x\in [a,b]$, $C_k (1\leq k\leq 6)$ be certain positive constants,
	\begin{align*}
			&\frac{C_1}{R}\leq|f''(x)|\leq\frac{C_2}{R},\\
			&\beta_k(x)=\frac{f^{(k)}(x)}{f^{''}(x)}, \quad|\beta_k|\leq\frac{C_k}{U^{k-2}},\ U\geq 1, 3\leq k\leq 5.
	\end{align*}
	
		Assume that $|3\beta_4(x)-5\beta_3^2(x)|\geq C_6U^{-2}$ and $\alpha\leq f'(x)\leq \beta$ for $a \leq x\leq b$. Then
		\begin{align*}
			\sum_{a < x\leq b}e\big(f(n)\big)=&\sum_{\alpha\leq u\leq \beta}\frac{b_u}{\sqrt{f''(n_u)}}e\big(f(n_u)-un_u+\frac{1}{8}\big)\\
			&+O\big(\log(2+(b-a)R^{-1})+(b-a+R)U^{-1}\big)\\
			&+O\Big(\min\Big\{\sqrt{R},\  \max\Big(\frac{1}{\langle\alpha\rangle},\frac{1}{\langle\beta\rangle}\Big)\Big\}\Big),
		\end{align*}
		where $n_u$ is the solution of $f'(n)=u$,
		\begin{align*}
			&\langle t\rangle =\left\{\begin{array}{ll}
				\|t\|,&\text{if }t\text{  is not an integer},\\
				\beta-\alpha,& \text{if } t\text{  is an integer},
			\end{array}\right.\\
			&b_u=\left\{\begin{array}{ll}
				1,&\text{if }\alpha< u<\beta,\\
				\frac{1}{2},& \text{if }u=\alpha \text{ or}~\beta\text{  is an integer},
			\end{array}\right.\\
			&\sqrt{f''}=\left\{\begin{array}{ll}
				\sqrt{f''},&\text{if }f''>0,\\
				i\sqrt{|f''|},& \text{if }f''<0.
			\end{array}\right.\\
		\end{align*}
\end{lemma}
\begin{lemma}\label{lem:AgDp}
	For any algebraic number $\alpha$,	not rational, the inequality 
	\[ \|n\alpha\|\geq c(\alpha,\varepsilon)n^{-1-\varepsilon} \]
	holds for all integers $n\geq 1$, where $c(\alpha,\varepsilon)$ is a positive constants.
\end{lemma}

\section{ A Voronoi type formula of $\Delta_{a,b}(x)$}

Let $T\geq 10$, $H\geq 2$ be large parameters satisfying  $\log H\ll\log T$. Suppose that $T\leq x\leq2T$. It is well-known that (see e.g. \cite{richert1952Uber})
\begin{equation}\label{Delta}
	\Delta_{a,b}(x)=f(a,b; x)+f(b,a;x)+O(1),
\end{equation}
where
\begin{equation*}
	f(a,b; x)=-\sum_{m\leq x^\frac{1}{a+b}}\psi\Big(\frac{x^\frac{1}{a}}{m^\frac{b}{a}}\Big).
\end{equation*}

By Lemma \ref{lem:1}, we have
\begin{equation}\label{fab}
	f(a,b; x)=\mathcal{R}(a,b; x)+G(a,b; x),
\end{equation}
with
\begin{align}
	\mathcal{R}(a,b; x) &= \frac{1}{2\pi i}\sum_{1\leq |h|\leq aH}\frac{1}{h}\sum_{m\leq x^\frac{1}{a+b}}e\Big(\frac{hx^\frac{1}{a}}{m^\frac{b}{a}}\Big), \notag\\
	G(a,b; x) &= O\Bigg(\sum_{m\leq x^\frac{1}{a+b}}\min\Big(1,\frac{1}{H\Big\|\frac{x^\frac{1}{a}}{m^\frac{b}{a}}\Big\|}\Big)\Bigg).\label{Gab}
\end{align}
Let $c=(2ab)^{ab}$, $J=\max\Bigl\{\big[\frac{\mathcal{L}-(a+b)\log\mathcal{L}}{(a+b)\log c}\big],\big[\frac{b\log(bH)}{(a+b)\log c}\big]\Bigr\}$. Define
\begin{align*}
	m_{h,0}& =x^\frac{1}{a+b}, \\ 
	m_{h,j}& =x^\frac{1}{a+b}\big(\frac{b}{a}h\big[\frac{b}{a}hc^{\frac{a+b}{a}j}\big]^{-1}\big)^\frac{a}{a+b},\ \text{ for} \ 1\leq j\leq J+1.	
\end{align*}
Then,
 \begin{equation*}
		c^{-1}\max\bigl\{T^\frac{1}{a+b}\mathcal{L}^{-1},(bH)^\frac{b}{a+b}\bigr\}<c^{J}\leq \max\bigl\{T^\frac{1}{a+b}\mathcal{L}^{-1},(bH)^\frac{b}{a+b}\bigr\},\qquad \text{and}\quad	m_{h,J}\ll\mathcal{L},
\end{equation*}
and moreover,
\begin{align}
	\mathcal{R}(a,b; x)&=\frac{1}{2\pi i}\sum_{1\leq |h|\leq aH}\frac{1}{h}\sum_{j=0}^{J}\sum_{m_{h, j+1}<m\leq m_{h, j}}e\Big(\frac{hx^\frac{1}{a}}{m^\frac{b}{a}}\Big)+O{(\mathcal{L}^2)}\notag\\
	&=-\frac{\Sigma_0}{2\pi i}+\frac{\overline{\Sigma_0}}{2\pi i}+O{(\mathcal{L}^2)},\label{R}
\end{align}
where \[ \Sigma_0=\sum_{1\leq h\leq aH}\frac{1}{h}\sum_{j=0}^{J}\sum_{m_{h, j+1}<m\leq m_{h, j}}e\Big(-\frac{hx^\frac{1}{a}}{m^\frac{b}{a}}\Big). \]

Set \[ S_{h,j}=\sum_{m_{h, j+1}<m\leq m_{h, j}}e\Big(-\frac{hx^\frac{1}{a}}{m^\frac{b}{a}}\Big), \qquad\text{for } j=1,2,\cdots J. \]
By Lemma \ref{lem:liu}, we get
\begin{equation*}
\begin{split}
	S_{h,0}=c_1(a,b)x^\frac{1}{2(a+b)}\mathop{{\sum}'}_{\frac{b}{a}h\leq r\leq [\frac{b}{a}hc^\frac{a+b}{a}]}h^\frac{a}{2(a+b)}r^{-\frac{2a+b}{2(a+b)}} e\Big(-c_2(a,b)x^\frac{1}{a+b}(h^ar^b)^\frac{1}{a+b}-\frac{1}{8}\Big)+O(\mathcal{L})\\
	+O\Bigg(\min\bigg\{h^{-\frac{1}{2}}x^\frac{1}{2(a+b)},\max\Big(\frac{1}{<\frac{b}{a}h>}, \frac{1}{[\frac{b}{a}hc^\frac{a+b}{a}]}\Big)\Bigg\}\Bigg),
\end{split}
\end{equation*}
and for $1\leq j\leq J$,
\begin{equation*}
\begin{split}
	S_{h,j}=c_1(a,b)x^\frac{1}{2(a+b)}\mathop{{\sum}'}_{[\frac{b}{a}hc^{\frac{a+b}{a}j}]\leq r\leq [\frac{b}{a}hc^{\frac{a+b}{a}(j+1)}]}h^\frac{a}{2(a+b)}r^{-\frac{2a+b}{2(a+b)}}e\Big(-c_2(a,b)x^\frac{1}{a+b}(h^ar^b)^\frac{1}{a+b}-\frac{1}{8}\Big)\\
	+O(\mathcal{L}),
\end{split}	
\end{equation*}
 where
\begin{align*}
	c_1(a,b)=a^\frac{b}{2(a+b)}b^\frac{a}{2(a+b)}(a+b)^{-\frac{1}{2}},\ \ 
	c_2(a,b)=\big(\frac{a}{b}\big)^\frac{b}{a+b}+\big(\frac{b}{a}\big)^\frac{a}{a+b}.
\end{align*}
Since $c^{J}\gg T^\frac{1}{a+b}\mathcal{L}^{-1}$,  the last term of the summation in $S_{h,J}$, corresponding to $r= [\frac{b}{a}hc^{\frac{a+b}{a}(J+1)}]$, is $o(x^{-\frac{1}{2(a+b)}})$.  Consequently, the range of $r$ in the summation can be replaced by $[\frac{b}{a}hc^{\frac{a+b}{a}J}]\leq r\leq \frac{b}{a}hc^{\frac{a+b}{a}(J+1)}$, without changing the validity of the equality for $S_{h,J}$. 

Inserting these into $\Sigma_0$,  we obtain
\begin{equation*}
	\begin{split}
		\Sigma_0=c_1(a,b)x^\frac{1}{2(a+b)}\sum_{1\leq h\leq aH}\mathop{{\sum}'}_{\frac{bh}{a}\leq r\leq \frac{b}{a}hc^{\frac{a+b}{a}(J+1)}}h^{-\frac{a+2b}{2(a+b)}}r^{-\frac{2a+b}{2(a+b)}} e\big(-c_2(a,b)x^\frac{1}{a+b}(h^ar^b)^\frac{1}{(a+b)}-\frac{1}{8}\big)\\
		+O\big(E(a,b; x, H)\big),
	\end{split}
\end{equation*}
where
\[E(a,b; x, H)=\sum_{1\leq h\leq aH}\frac{1}{h}\min\bigg\{h^{-\frac{1}{2}}x^\frac{1}{2(a+b)},\max\Big(\frac{1}{<\frac{b}{a}h>}, \frac{1}{[\frac{b}{a}hc^\frac{a+b}{a}]}\Big)\bigg\}+\mathcal{L}^3.\]
From \eqref{R},  by using $c^{J+1}>(bH)^\frac{b}{a+b}$,  we have
\begin{align}
		\mathcal{R}(a,b;x)&=\frac{c_1(a,b)}{\pi}x^\frac{1}{2(a+b)}\sum_{1\leq h\leq aH}\mathop{{\sum}'}_{\frac{bh}{a}\leq r\leq \frac{b}{a}hc^{\frac{a+b}{a}(J+1)}}h^{-\frac{a+2b}{2(a+b)}}r^{-\frac{2a+b}{2(a+b)}}\qquad\qquad\notag\\
		&\qquad\qquad\qquad\times \cos\big(2\pi c_2(a,b)x^\frac{1}{a+b}(h^ar^b)^\frac{1}{a+b}-\frac{\pi}{4}\big)+O\big(E(a,b; x, H)\big)\notag\\
	&=\mathcal{R}^*(a,b;x,H)+\mathcal{R}_1(a,b; x)+O\big(E(a,b; x, H)\big),\label{Rab}
\end{align}
where
\begin{gather}
	\mathcal{R}^*(a,b;x,H)=\frac{c_1(a,b)}{\pi}x^\frac{1}{2(a+b)}\!\!\sum_{1\leq h\leq aH}\mathop{{\sum}}_{\frac{bh}{a}\leq r\leq bH}\!\! b_rh^{-\frac{a+2b}{2(a+b)}}r^{-\frac{2a+b}{2(a+b)}}\cos\big(2\pi c_2(a,b)x^\frac{1}{a+b}(h^ar^b)^\frac{1}{a+b}\!-\!\frac{\pi}{4}\big),\notag\\
	\mathcal{R}_1(a,b; x)=\frac{c_1(a,b)}{\pi}x^\frac{1}{2(a+b)}\!\!\sum_{1\leq h\leq aH}\mathop{{\sum}}_{bH< r\ll hc^{\frac{a+b}{a}J}}\!\! h^{-\frac{a+2b}{2(a+b)}}r^{-\frac{2a+b}{2(a+b)}}\cos\big(2\pi c_2(a,b)x^\frac{1}{a+b}(h^ar^b)^\frac{1}{a+b}\!-\!\frac{\pi}{4}\big),	\label{R1ab} 
\end{gather}
with $b_{r}=\frac{1}{2}$ if $r=\frac{bh}{a}$ is integer, and $b_{r}=1$ for $\frac{bh}{a}<r\leq  bH$.

Similarly,  we have
\begin{equation}\label{fba}
	f(b,a;x)=\mathcal{R}(b,a; x)+G(b,a).
\end{equation}
Here
\begin{align}
		\mathcal{R}(b,a; x)&=\frac{c_1(b,a)}{\pi}x^\frac{1}{2(a+b)}\sum_{1\leq h\leq bH}\mathop{{\sum}'}_{\frac{ah}{b}\leq r\leq \frac{a}{b}hc^{\frac{a+b}{b}(J+1)}}h^{-\frac{2a+b}{2(a+b)}}r^{-\frac{a+2b}{2(a+b)}}\notag\\
		&\qquad\qquad\qquad\times \cos\big(2\pi c_2(b,a)x^\frac{1}{a+b}(h^br^a)^\frac{1}{a+b}-\frac{\pi}{4}\big)+O\big(E(b,a; x, H)\big)\notag\\
	  &=\frac{c_1(b,a)}{\pi}x^\frac{1}{2(a+b)}\sum_{1\leq r\leq bH}\mathop{{\sum}'}_{\frac{ar}{b}\leq h\leq \frac{a}{b}rc^{\frac{a+b}{b}(J+1)}}h^{-\frac{a+2b}{2(a+b)}}r^{-\frac{2a+b}{2(a+b)}}\notag\\
	  &\qquad\qquad\qquad\times \cos\big(2\pi c_2(b,a)x^\frac{1}{a+b}(h^ar^b)^\frac{1}{a+b}-\frac{\pi}{4}\big)+O\big(E(b,a; x, H)\big)\notag\\
	  &=\mathcal{R}^*(b,a;x,H)+\mathcal{R}_1(b,a; x)+O\big(E(b,a; x, H)\big),\label{Rba}
\end{align}
\begin{equation*}\label{Gba} 
	G(b,a; x) = O\Bigg(\sum_{m\leq x^\frac{1}{a+b}}\min\Big(1,\frac{1}{H\Big\|\frac{x^\frac{1}{b}}{m^\frac{a}{b}}\Big\|}\Big)\Bigg),
\end{equation*}
and
\begin{align*}
	&\mathcal{R}^*(b,a;x,H)=\frac{c_1(b,a)}{\pi}x^\frac{1}{2(a+b)}\!\!\sum_{1\leq r\leq bH}\mathop{{\sum}}_{\small\frac{ar}{b}\leq h\leq aH}\!\! b_hh^{-\frac{a+2b}{2(a+b)}}r^{-\frac{2a+b}{2(a+b)}}\cos\big(2\pi c_2(b,a)x^\frac{1}{a+b}(h^ar^b)^\frac{1}{a+b}\!-\!\frac{\pi}{4}\big),\notag\\
	&\mathcal{R}_1(b,a; x)=\frac{c_1(b,a)}{\pi}x^\frac{1}{2(a+b)}\!\!\sum_{1\leq r\leq bH}\mathop{{\sum}}_{aH< h\ll rc^{\frac{a+b}{b}J}}\!\! h^{-\frac{a+2b}{2(a+b)}}r^{-\frac{2a+b}{2(a+b)}}\cos\big(2\pi c_2(b,a)x^\frac{1}{a+b}(h^ar^b)^\frac{1}{a+b}\!-\!\frac{\pi}{4}\big),\label{R1ba}\\
	&E(b,a; x, H)=\sum_{1\leq h\leq bH}\frac{1}{h}\min\Big(h^{-\frac{1}{2}}x^\frac{1}{2(a+b)},\max\Big(\frac{1}{<\frac{a}{b}h>}, \frac{1}{[\frac{a}{b}hc^\frac{a+b}{b}]}\Big)\Big)+\mathcal{L}^3,\notag
\end{align*}
where $b_h=\frac{1}{2}$ if $h=\frac{ar}{b}$ is integer, and $b_{h}=1$ for $\frac{ar}{b}< h\leq aH$.

Denote
\begin{equation}\label{sR*}
	\Delta_{a,b}^*(x,H)=\mathcal{R}^*(a,b;x,H)+\mathcal{R}^*(b,a;x,H). 
\end{equation}
Using the symmetry $c_1(a,b)=c_1(b,a)$ and $c_2(a,b)=c_2(b,a)$, we see
\begin{equation}\label{Dlt*}
	\Delta_{a,b}^*(x,H)=\frac{c_1(a,b)}{\pi}x^\frac{1}{2(a+b)}\sum_{h=1}^{ aH} \sum_{r=1}^{bH}h^{-\frac{a+2b}{2(a+b)}}r^{-\frac{2a+b}{2(a+b)}}\cos\big(2\pi c_2(a,b)x^\frac{1}{a+b}(h^ar^b)^\frac{1}{a+b}-\frac{\pi}{4}\big).
\end{equation}
Write
\begin{equation}\label{SE}
	E_{a,b}(x,H)=E(a,b; x, H)+E(b,a; x, H). 
\end{equation}
	 For $a, b\in  \mathbb{Z}^+$,  we have 
	 \begin{equation}\label{eabq}
	 	E_{a,b}(x,H)\ll \mathcal{L}^3.
	 \end{equation}
	For $\dfrac{a}{b}\in \mathbb{\overline{Q}}\setminus \mathbb{Q}$, by applying Lemma \ref{lem:AgDp}, we get
	\begin{equation}\label{eabr}
		\begin{split}
			E_{a,b}(x,H)&\ll \sum_{1\leq h\leq bH}\frac{1}{h}\min\Big(h^{-\frac{1}{2}}x^\frac{1}{2(a+b)},\frac{1}{\|\frac{a}{b}h\|}+\frac{1}{\|\frac{b}{a}h\|}
			 \Big)+\mathcal{L}^3\\
			 &\ll \sum_{1\leq h\leq bH}\frac{1}{h}\min\Big(h^{-\frac{1}{2}}x^\frac{1}{2(a+b)},h^{1+\varepsilon} \Big)+\mathcal{L}^3\\
			 &\ll \sum_{1\leq h\leq bH}\frac{1}{h}\min\Big(h^{-\frac{1}{2}}x^\frac{1}{2(a+b)}\Big)^\frac{2}{3}h^{\frac{1}{3}(1+\varepsilon)}+\mathcal{L}^3\\
			 &\ll x^\frac{1}{3(a+b)}H^\varepsilon.
		\end{split}
		\end{equation}

Combining \eqref{Delta} \eqref{fab},\eqref{Rab},\eqref{fba}, \eqref{Rba}, \eqref{sR*} and \eqref{SE}, we obtain that
\begin{multline}\label{Dlt}
	\Delta_{a,b}(x)=\Delta_{a,b}^*(x,H)+\mathcal{R}_1(a,b; x)+\mathcal{R}_1(b,a; x)+G(a,b;x)+G(b,a;x)	 +O\big(E_{a,b}(x,H)\big),
\end{multline}
which can be viewed as a Voronoi type formula of $\Delta_{a,b}(x)$.

\section{On a series and a special sum}\label{sr&sm}

Suppose that $a,b>0$ are fixed real numbers, $T\geq 10, H\geq 2, R\geq 2$ are large parameters satisfying  $\log H\ll\log T$ and $\log R\ll\log T.$ 
Denote

\begin{align*}
\Sigma_1(H;a,b)&=\sum_{\substack{h_1>H;\\ h_1^ar_1^b=h_2^ar_2^b}}(h_1h_2)^{-\frac{a+2b}{2(a+b)}}(r_1r_2)^{-\frac{2a+b}{2(a+b)}},\\
\Sigma_2(T; H, R, a,b)&=\sum_{\substack{h_i\leq H, r_i\leq R, i=1,2;\\{\small 0<|\eta(h_1,r_1,h_2,r_2)|<\frac{1}{10}(h_1^ar_1^bh_2^ar_2^b)^\frac{1}{2(a+b)}}}}(h_1h_2)^{-\frac{a+2b}{2(a+b)}}(r_1r_2)^{-\frac{2a+b}{2(a+b)}}\\
	&\qquad\qquad\qquad\qquad\qquad\qquad\times\min\big(T^\frac{1}{a+b},\frac{1}{|\eta(h_1,r_1,h_2,r_2)|}\big),
\end{align*}
where
 $$\eta(h_1,r_1,h_2,r_2)=(h_1^ar_1^b)^\frac{1}{a+b}-(h_2^ar_2^b)^\frac{1}{a+b}.$$

In this section, we shall study $\Sigma_1(H;a,b)$ and $\Sigma_2(T; H,R, a,b)$, and get the following estimates.
\begin{lemma}\label{lem:S1Hab}
	\begin{equation*}
		\Sigma_1(H;a,b) \ll H^{-\frac{b}{2(a+b)}}\log^2(2H).
	\end{equation*}
\end{lemma}

\begin{remark}\label{lem:gab}
	Due to the symmetry of $a$ and $b$, we see that
	\begin{equation*}
		\Sigma_1(R;b,a)=\sum_{\substack{r_1>R;\\ h_1^ar_1^b=h_2^ar_2^b}}(h_1h_2)^{-\frac{a+2b}{2(a+b)}}(r_1r_2)^{-\frac{2a+b}{2(a+b)}}\ll R^{-\frac{a}{2(a+b)}}\log^2(2R).
	\end{equation*}
	 It is following that the infinity series 
	\begin{equation*}
		G_{a,b}=\mathop{\sum_{h_1=1}^{\infty}\sum_{h_2=1}^{\infty}\sum_{j_1=1}^{\infty}\sum_{j_2=1}^{\infty}}_{h_1^ar_1^b=h_2^ar_2^b}(h_1h_2)^{-\frac{a+2b}{2(a+b)}}(r_1r_2)^{-\frac{2a+b}{2(a+b)}}
	\end{equation*}
	 is convergent.
\end{remark}

\begin{lemma}\label{lem:Sab2}
	(I) If $\frac{a}{b}\in \mathbb{\overline{Q}}\setminus \mathbb{Q}$, then there exists  a positive constant $c_4(a,b)$, such that
	\begin{equation*}
		\Sigma_2(T; H,R, a,b)+\Sigma_2(T; H,R, b,a)
		\ll T^{\frac{1}{a+b}}\exp\big\{-c_4(a,b)(\log T)^\frac{1}{2}(\log\log T)^{-\frac{1}{2}}\big\};
	\end{equation*}
	
	(II) If $a, b\in  \mathbb{Z}^+$ with $b\geq a$ and $(a,b)=1$, then
	\begin{equation*}
		\Sigma_2(T; H,R, a,b)+\Sigma_2(T; H,R,b,a) \ll T^{\frac{1}{a+b}-\frac{a}{b(a+b)(a+b-1)}}\mathcal{L}^4.
	\end{equation*}
\end{lemma}

To prove these results, we need the following lemmas. Lemma \ref{lem:Dph} corresponds to Lemma 5.1 in Zhai and Cao \cite{zhai2010mean}, with a little modification  in expression where we replace  "$\geq 1$" in their lemma by "$\geq \frac{1}{2}$". Lemma \ref{lem:baker} follows from Theorem 2 of  \cite{baker1967Linear}. Lemma \ref{lem:loglf} is Theorem 10.21 of Waldschmidt \cite{waldschmidt2000Diophantine}, see also \cite[Theorem 1.1]{philippon1988Lower}. Lemma \ref{lem:etal} can be used to get a lower bound for $\eta(h_1,r_1,h_2,r_2)$.

\begin{lemma}\label{lem:Dph}
	Let $\mu$ and $\nu$ be fixed non-zero real numbers. Suppose $H_1\geq \frac{1}{2}$, $H_2\geq \frac{1}{2}$, $R_1\geq \frac{1}{2}$, $R_2\geq \frac{1}{2}$ and $\delta>0$ are  real numbers. Denote
	\begin{equation*}
		\mathcal{A}(H_1,H_2,R_1,R_2; \delta)=\#\{(h_1, h_2, r_1, r_2)||h_1^\mu r_1^\nu-h_2^\mu r_2^\nu|\leq \delta, h_i\!\sim\! H_i, r_i\!\sim \!R_i, i=1,2\}.
	\end{equation*}
	Then
	$$\mathcal{A}(H_1, H_2, R_1,R_2; \delta)\ll \delta (H_1H_2)^{1-\frac{\mu}{2}}(R_1R_2)^{1-\frac{\nu}{2}}+(H_1H_2R_1R_2)^\frac{1}{2}\log^2( 2H_1H_2R_1R_2),$$
	where $\ll$-constant depends only on $\mu, \nu$. 
\end{lemma}

\begin{lemma}\label{lem:baker}
	Let $m \geq 1$ be a positive integer, $\alpha_1, \ldots, \alpha_m$ nonzero algebraic numbers, $\lambda_1, \ldots, \lambda_m$ logarithms of $\alpha_1, \ldots, \alpha_m$ respectively. If $\lambda_1, \ldots, \lambda_m$ are linearly independent over $\mathbb{Q}$, then they are linealy independent over $\mathbb{\overline{Q}}$, too.
\end{lemma}

\begin{lemma}\label{lem:loglf}
	Let $m$, $\alpha_1, \ldots, \alpha_m$, $\lambda_1, \ldots, \lambda_m$ be as stated in Lemma \ref{lem:baker}, $\beta_1, \ldots, \beta_m$ algebraic numbers,  $D$ the degree over $\mathbb{Q}$ of the number field
	\(\mathbb{Q}(\alpha_1, \ldots, \alpha_m, \beta_1, \ldots, \beta_m).\) Assume that the number
	\(\Lambda = \beta_1 \lambda_1 + \cdots + \beta_m \lambda_m	\)
	is nonzero.	
	
	Let $A_1, \ldots, A_m, B, E$ and $E^*$ be positive numbers with $E \geq e$,	$E^*=\max\{E, D\}$ and
	\[	\log A_i = \max \left\{ h(\alpha_i), \frac{E}{D} |\lambda_i|, \frac{m\log E}{D} \right\} \quad (1 \leq i \leq m),\]
	\[\log B \geq \max_{0 \leq i \leq m} h(\beta_i) \quad \text{and} \quad B \geq \max\{E^*, \log A_1, \ldots, \log A_m\},	\]
	where $h(\alpha)$ is the (Weil) absolute logarithmic height of an algebraic number $\alpha$.
	Then
	\[	|\Lambda| \geq \exp \left\{ -C^m m^{2m} D^{m+2} (\log A_1) \cdots (\log A_m) (\log B) (\log E^*) (\log E)^{-m-1} \right\}.	\]
\end{lemma}

%
%
%

\begin{lemma}\label{lem:etal}
 Let $h_1,h_2,r_1,r_2\in \mathbb{Z}^+$, $\alpha, \beta$ are positive algebric numbers such that $\frac{\beta}{\alpha}\in \mathbb{\overline{Q}}\setminus \mathbb{Q}$. Denote $M=\max\{h_1,h_2,r_1,r_2\}$. If $h_1^\alpha r_1^\beta-h_2^\alpha r_2^\beta\neq 0$, then 
\begin{equation*}
	|h_1^\alpha r_1^\beta-h_2^\alpha r_2^\beta|> \exp\{-C(\alpha,\beta)(\log M)^{2}\log\log M\},
\end{equation*}
where $C(\alpha,\beta)$ is a positive constant depending only on $\alpha, \beta$.
\end{lemma}
\begin{proof}[Proof of Lemma \ref{lem:etal}]
	Assume that $h_1^\alpha r_1^\beta>h_2^\alpha r_2^\beta$ without loss of generality. Then
	\begin{equation}\label{Dhr0}
		h_1^\alpha r_1^\beta-h_2^\alpha r_2^\beta=h_2^\alpha r_2^\beta\bigg(\Big(\frac{h_1}{h_2}\Big)^\alpha \Big(\frac{r_1}{r_2}\Big)^\beta-1\bigg)>\alpha\log \frac{h_1}{h_2}+\beta\log\frac{r_1}{r_2}>0,
	\end{equation}
	by using $x>\log (1+x)$ for $x>0$.
	
	If $h_1=h_2$, then $r_1>r_2$, and it follows that
	\[ h_1^\alpha r_1^\beta-h_2^\alpha r_2^\beta>\beta\log\frac{r_1}{r_2}.\]
	Since $1\leq r_1\neq r_2\leq M$, we see
	\begin{equation}\label{logl}
		\Big|\log\frac{r_1}{r_2}\Big|\geq \log\Big(1+\frac{1}{M}\Big)\geq  \frac{1}{M} \log2.
	\end{equation}
	Thus
	\[ h_1^\alpha r_1^\beta-h_2^\alpha r_2^\beta>\beta (\log2)M^{-1}.\]
	Similarly, if $r_1=r_2$, then $h_1>h_2$ and 
	\[ h_1^\alpha r_1^\beta-h_2^\alpha r_2^\beta>\alpha(\log2)M^{-1}.\]
	So Lemma \ref{lem:etal} holds for these two cases.
	
	For the remainder of the proof, we assume $h_1\neq h_2$ and $ r_1\neq r_2$. 
	
	If $\log \frac{h_1}{h_2}\big(\log\frac{r_1}{r_2}\big)^{-1}\not\in \mathbb{Q}$, then  from  Lemma \ref{lem:baker}, we see  $\alpha\log \frac{h_1}{h_2}+\beta\log\frac{r_1}{r_2}\neq 0$.
	
	In Lemma \ref{lem:loglf}, let $E=e$.  Noting that $$h\big(\frac{h_1}{h_2}\big)=\max \big\{\log h_1, \log h_2\big\} \text{and} h\big(\frac{r_1}{r_2}\big)=\max \big\{\log r_1, \log r_2\big\},$$  
	we obtain
	\begin{align*}
		&\log A_1= \max \left\{ \log h_1, \log h_2, \frac{e}{D} \Big|\log \frac{h_1}{h_2}\Big|,            \frac{2}{D} \right\}\ll \log M,\\
		&\log A_2= \max \left\{ \log r_1, \log r_2, \frac{e}{D} \Big|\log \frac{r_1}{r_2}\Big|, \frac{2}{D} \right\}\ll \log M.
	\end{align*}
	Choosing $B$ such that
	\[ \log B=\max \left\{ h(\alpha), h(\beta), \log(e+D), \log\log A_1, \log\log A_2\right\}, \]
	we have $\log B\ll\log\log M$. Consequently,
	\[	|\Lambda| =\alpha\log \frac{h_1}{h_2}+\beta\log\frac{r_1}{r_2}>\exp\{-C(\alpha,\beta)(\log M)^{2}\log\log M\}. \]
 From \eqref{Dhr0}, we can derive the  conclusion of Lemma \ref{lem:etal}.
	
	If $\log \frac{h_1}{h_2}\big(\log\frac{r_1}{r_2}\big)^{-1}\in \mathbb{Q}$, then there exists $p\in \mathbb{Z}^+$ and $q\in\mathbb{Z}\setminus\{0\}$ such that 
	$$\log \frac{h_1}{h_2}\big(\log\frac{r_1}{r_2}\big)^{-1}=\frac{q}{p}.$$ 
	By Lemma \ref{lem:AgDp}, we get
	\[ \alpha\log \frac{h_1}{h_2}+\beta\log\frac{r_1}{r_2}=\frac{\alpha}{p}\Big|\log\frac{r_1}{r_2}\Big|\Big|q+\frac{\beta}{\alpha}p\Big|\geq \frac{\alpha}{p}\Big|\log\frac{r_1}{r_2}\Big|c(\frac{\beta}{\alpha},\varepsilon)p^{-1-\varepsilon}. \]
	Combing this with \eqref{Dhr0} and \eqref{logl}, we have 
	\begin{equation}\label{Dhr1}
		h_1^\alpha r_1^\beta-h_2^\alpha r_2^\beta>\alpha(\log 2) c(\frac{\beta}{\alpha},\varepsilon)M^{-1} p^{-2-\varepsilon}. 
	\end{equation}

	Let $h_{10}, h_{20}, r_{10}, r_{20}$ be positive integers such that $(h_{10}, h_{20})=1$, $(r_{10}, r_{20})=1$ and $\frac{h_{10}}{h_{20}}=\frac{h_1}{h_2}$, $\frac{r_{10}}{r_{20}}=\frac{r_1}{r_2}$. Then
	\[ \Big(\frac{h_{10}}{h_{20}}\Big)^p=\Big(\frac{r_{10}}{r_{20}}\Big)^q. \]
	If $q>0$, we see
	 \[ {h_{10}}^p{r_{20}}^q={h_{20}}^p{r_{10}}^q,\]
	from which we deduce that	
	\[ {h_{10}}^p={r_{10}}^q \text{  \ and } \ {h_{20}}^p={r_{20}}^q.\]
	If $q<0$, then	
	\[ {h_{10}}^p{r_{10}}^{|q|}={h_{20}}^p{r_{20}}^{|q|},  \]
	and similarly,
	 \[ {h_{10}}^p={r_{20}}^{|q|} \text{  \ and } {h_{20}}^p={r_{10}}^{|q|}.\]
	By the fundamental theorem of arithmetic, noting that $(p,q)=1$, we get
	\[ {h_{i0}}^\frac{1}{|q|}\in \mathbb{Z}^+\quad\text{and}\quad {r_{i0}}^\frac{1}{p}\in \mathbb{Z}^+ \qquad \text{for } i=1,2.\]
	Since  $\frac{r_{10}}{r_{20}}=\frac{r_1}{r_2}\neq 1$, it follows that $r_{10}\neq r_{20}$. Therefore, \(\max\{r_{10}, r_{20}\}\geq 2\), and consequently,
	\[ \max\big\{{r_{10}}^\frac{1}{p}, {r_{20}}^\frac{1}{p}\big\}\geq 2. \]
	Using the fact that $r_{i0}\leq r_i\leq M$ for $i=1,2$, we obtain
	 $$p \leq \frac{\log M}{\log 2}.$$
	
	Substituting into  \eqref{Dhr1}, we derive that
	\begin{equation*}
		h_1^\alpha r_1^\beta-h_2^\alpha r_2^\beta>\alpha(\log 2)^{3+\varepsilon}\, c(\frac{\beta}{\alpha},\varepsilon)M^{-1} (\log M)^{-2-\varepsilon}. 
	\end{equation*}
	Thus, the conclusion of Lemma \ref{lem:etal} follows, completing the proof for all cases.
\end{proof}

\begin{proof}[Proof of Lemma \ref{lem:S1Hab}] 
	Set
\begin{equation*}
	\Sigma_{a,b}(H_1,H_2,R_1, R_2; \delta)=\sum_{\substack{h_i\sim H_i, r_i\sim R_i, i=1,2;\\ |h_1^ar_1^b-h_2^ar_2^b|\leq \delta}} (h_1h_2)^{-\frac{a+2b}{2(a+b)}}(r_1r_2)^{-\frac{2a+b}{2(a+b)}},
\end{equation*}
where $ H_i\geq \frac{1}{2}$,  $R_i\geq \frac{1}{2}$ for $i=1,2$. 
By using Lemma \ref{lem:Dph} with $\mu=a$ and $\nu=b$, we see
\begin{align*}
	&\quad\ \ \Sigma_{a,b}(H_1,H_2,R_1, R_2; \delta)\ll (H_1H_2)^{-\frac{a+2b}{2(a+b)}}(R_1R_2)^{-\frac{2a+b}{2(a+b)}}\mathcal{A}(H_1, H_2, R_1,R_2; \delta)\\
	&\ll (H_1H_2)^{-\frac{a+2b}{2(a+b)}}(R_1R_2)^{-\frac{2a+b}{2(a+b)}}\big(\delta(H_1H_2)^{1-\frac{a}{2}}(R_1R_2)^{1-\frac{b}{2}}+(H_1H_2R_1R_2)^\frac{1}{2}\log^2( 2H_1H_2R_1R_2)\big).
\end{align*}
Let $\delta \to 0$, we get
\begin{equation*}
	\begin{split}
		\Sigma_{a,b}(H_1,H_2,R_1, R_2; 0)&=\sum_{\substack{h_i\sim H_i, r_i\sim R_i, i=1,2;\\ h_1^ar_1^b=h_2^ar_2^b}} (h_1h_2)^{-\frac{a+2b}{2(a+b)}}(r_1r_2)^{-\frac{2a+b}{2(a+b)}}\\
		&\ll (H_1H_2)^{-\frac{b}{2(a+b)}}(R_1R_2)^{-\frac{a}{2(a+b)}}\log^2( 2H_1H_2R_1R_2).
	\end{split}
\end{equation*}
Hence,
	\begin{equation*}
	\Sigma_1(H;a,b)=\sum_{j_1=0}^{\infty}\sum_{j_2=-1}^{\infty}\sum_{j_3=-1}^{\infty}\sum_{j_4=-1}^{\infty}\Sigma_{a,b}(2^{j_1}H, 2^{j_2}, 2^{j_3},2^{j_4}; 0)\ll H^{-\frac{b}{2(a+b)}}\log^2(2H),
\end{equation*}
which is  Lemma  \ref{lem:S1Hab}.
\end{proof}
\begin{proof}[Proof of of Lemma \ref{lem:Sab2}]
	We first consider (I) and suppose that $\frac{a}{b}\in \mathbb{\overline{Q}}\setminus \mathbb{Q}$. Let
	\begin{equation*}
		\begin{split}
			U_{a,b}(T; H_1, H_2, R_1,R_2)=\!\!\sum_{\substack{h_i\sim H_i, r_i\sim R_i, i=1,2;\\{\small 0<|\eta(h_1,r_1,h_2,r_2)|<\frac{1}{10}(h_1^ar_1^bh_2^ar_2^b)^\frac{1}{2(a+b)}}}}\!\!(h_1h_2)^{-\frac{a+2b}{2(a+b)}}(r_1r_2)^{-\frac{2a+b}{2(a+b)}}&\\
			\qquad\times\min\big(T^\frac{1}{a+b},\frac{1}{|\eta(h_1,r_1,h_2,r_2)|}\big),&
		\end{split}	
	\end{equation*}
	with $\frac{1}{2}\leq H_i\leq H$,  $\frac{1}{2}\leq R_i\leq R$ ($i=1,2$). The condition $|\eta(h_1,r_1,h_2,r_2)|<\frac{1}{10}(h_1^ar_1^bh_2^ar_2^b)^\frac{1}{2(a+b)}$ gives 
	\begin{equation}\label{asymphr}
		H_1^aR_1^b\asymp H_2^aR_2^b.
	\end{equation}
	Let
	\begin{equation}\label{Uab}
		U_{a,b}(T; H_1, H_2, R_1,R_2)=U_1+U_2,
	\end{equation}
	where $U_1$ and $U_2$ represent the contributions corresponding to the cases $T^\frac{1}{a+b}<\frac{1}{|\eta(h_1,r_1,h_2,r_2)|}$ and $T^\frac{1}{a+b}\geq\frac{1}{|\eta(h_1,r_1,h_2,r_2)|}$, respectively, i.e.
	\begin{align*}
		U_1&=\sum_{\substack{h_i\sim H_i, r_i\sim R_i, i=1,2;\\{\small 0<|\eta(h_1,r_1,h_2,r_2)|<T^{-\frac{1}{a+b}}}}}(h_1h_2)^{-\frac{a+2b}{2(a+b)}}(r_1r_2)^{-\frac{2a+b}{2(a+b)}}T^\frac{1}{a+b},\\
		U_2&\ll \!\!\sum_{\substack{h_i\sim H_i, r_i\sim R_i, i=1,2;\\{\small T^{-\frac{1}{a+b}}\leq |\eta(h_1,r_1,h_2,r_2)|<\frac{1}{5}(H_1^aR_1^bH_2^aR_2^b)^\frac{1}{2(a+b)}}}}\!\!\!(h_1h_2)^{-\frac{a+2b}{2(a+b)}}(r_1r_2)^{-\frac{2a+b}{2(a+b)}}\frac{1}{|\eta(h_1,r_1,h_2,r_2)|}.
	\end{align*}
	
	Applying Lemma \ref{lem:Dph} with $\mu=\frac{a}{a+b}$ and $\nu=\frac{b}{a+b}$, we get
	\begin{equation}\label{u11}
		\begin{split}
		   U_1&\ll T^{\frac{1}{a+b}}(H_1H_2)^{-\frac{a+2b}{2(a+b)}}(R_1R_2)^{-\frac{2a+b}{2(a+b)}}\sum_{\substack{h_i\sim H_i, r_i\sim R_i, i=1,2;\\{\small 0<|\eta(h_1,r_1,h_2,r_2)|<T^{-\frac{1}{a+b}}}}}1\\
	   &\ll 1+T^{\frac{1}{a+b}}(H_1H_2)^{-\frac{b}{2(a+b)}}(R_1R_2)^{-\frac{a}{2(a+b)}}\mathcal{L}^2
\\
		   &\ll 1+T^{\frac{1}{a+b}}(H^a_1R_1^b)^{-\frac{a}{2b(a+b)}}(H^a_2R^b_2)^{-\frac{a}{2b(a+b)}}(H_1H_2)^{-\frac{b-a}{2b}}\mathcal{L}^2 \\
		   &\ll 1+T^{\frac{1}{a+b}}(H^a_1R_1^b)^{-\frac{a}{b(a+b)}}(H_1H_2)^{-\frac{b-a}{2b}}\mathcal{L}^2,
		\end{split}
	\end{equation}
	where we used  \eqref{asymphr} in the last step.
	
	Let $M=\max\{H_1, H_2, R_1, R_2\}$. By using Lemma \ref{lem:etal}, noting that  $\log H_i\ll\log T$, $\log R_i \ll\log T$ for i=1,2,  we get
	\begin{equation}\label{etal}
		|\eta(h_1,r_1,h_2,r_2)|>\exp\{-c_3(a,b)(\log M)^{2}\log\log T\},
	\end{equation}
	where $c_3(a,b)$ is a positive constant depending on $a, b$. Therefore, the condition
	 $$|\eta(h_1,r_1,h_2,r_2)|<T^{-\frac{1}{a+b}}$$  
	 suggests that 
	\[T^{\frac{1}{a+b}} < |\eta(h_1,r_1,h_2,r_2)|^{-1}<\exp\{c_3(a,b)(\log M)^{2}\log\log T\},\]
	which leads to
	\[\log T<(a+b)c_3(a,b)(\log M)^{2}\log\log T.\]
	Rearranging, we obtain
	\[(\log M)^{2}>(a+b)^{-1}(c_3(a,b))^{-1}\log T(\log\log T)^{-1},\]
	or equivalently,
	\begin{equation*}\label{mlb}
		M>\exp\{(a+b)^{-\frac{1}{2}}(c_3(a,b))^{-\frac{1}{2}}(\log T)^\frac{1}{2}(\log\log T)^{-\frac{1}{2}}\}.
	\end{equation*}
	Hence,  from  \eqref{u11}, we deduce that
	\begin{equation}\label{u1}
		\begin{split}
				U_1&\ll  1+T^{\frac{1}{a+b}}M^{-\frac{a}{a+b}}\mathcal{L}^2\\
				&\ll T^{\frac{1}{a+b}}\mathcal{L}^2\exp\{-a(a+b)^{-\frac{3}{2}}(c_3(a,b))^{-\frac{1}{2}}(\log T)^\frac{1}{2}(\log\log T)^{-\frac{1}{2}}\}.
		\end{split}
	\end{equation}
	
	Now we estimate $U_2$. 
	By using \eqref{etal},  the condition for $U_2$ can be written as
	\[ \max\big\{T^{-\frac{1}{a+b}}, \exp\{-c_3(a,b)(\log M)^{2}\log\log T\}\big\}\ll |\eta(h_1,r_1,h_2,r_2)|<\frac{1}{5}(H_1^aR_1^bH_2^aR_2^b)^\frac{1}{2(a+b)}. \]
	Let $\Delta=\frac{1}{5}(H_1^aR_1^bH_2^aR_2^b)^\frac{1}{2(a+b)}.$ We split the range of $|\eta(h_1,r_1,h_2,r_2)|$ into subintervals of the form $(2^{-(j+1)}\Delta, 2^{-j}\Delta] \text{ for } 0\leq j<J_0$, where $J_0$ satisfies that 
	\[ 2^{-J_0}\Delta\asymp \max\big\{T^{-\frac{1}{a+b}}, \exp\{-c_3(a,b)(\log M)^{2}\log\log T\}\big\}, \]
	which implies
	\begin{equation}\label{J0}
		2^{J_0}\Delta^{-1}\asymp \min\big\{T^{\frac{1}{a+b}}, \exp\{c_3(a,b)(\log M)^{2}\log\log T\}\big\} \text{\qquad and\qquad }J_0\ll \mathcal{L}^3.
	\end{equation}
	Then, 
	\begin{align*}
		\notag	U_2&\ll (H_1H_2)^{-\frac{a+2b}{2(a+b)}}(R_1R_2)^{-\frac{2a+b}{2(a+b)}}\sum_{j=0}^{J_0} \sum_{|\eta(h_1,r_1,h_2,r_2)|\sim 2^{-(j+1)}\Delta}\frac{1}{|\eta(h_1,r_1,h_2,r_2)|}\\
		\notag	&\ll (H_1H_2)^{-\frac{a+2b}{2(a+b)}}(R_1R_2)^{-\frac{2a+b}{2(a+b)}}\sum_{j=0}^{J_0}2^{j+1}\Delta^{-1} \mathcal{A}(H_1, H_2, R_1,R_2; 2^{-j}\Delta).
	\end{align*}
	By using Lemma \ref{lem:Dph}, with $\mu=\frac{a}{a+b}$ and $\nu=\frac{b}{a+b}$, we see
	\begin{equation*}
		\mathcal{A}(H_1, H_2, R_1,R_2; 2^{-j}\Delta)\ll2^{-j}\Delta(H_1H_2)^\frac{a+2b}{2(a+b)}(R_1R_2)^\frac{2a+b}{2(a+b)}+(H_1H_2R_1R_2)^\frac{1}{2}\mathcal{L}^2.
	\end{equation*}
	Substituting into the above inequality, by \eqref{asymphr} and \eqref{J0}, we get
	\begin{equation}\label{u20}
	\begin{split}
		U_2&\ll 2(J_0+1)+(H_1H_2)^{-\frac{b}{2(a+b)}}(R_1R_2)^{-\frac{a}{2(a+b)}}\Delta^{-1}(2^{J_0+2}-2)\mathcal{L}^2\\
		&\ll\mathcal{L}^3+(H_1H_2)^{-\frac{b}{2(a+b)}}(R_1R_2)^{-\frac{a}{2(a+b)}}\mathcal{L}^2 \min\big\{T^{\frac{1}{a+b}}, \exp\{c_3(a,b)(\log M)^{2}\log\log T\}\big\}\\
		&\ll\mathcal{L}^3+(H_1H_2)^{-\frac{b-a}{2b}}(H_1^aR_1^b)^{-\frac{a}{b(a+b)}}\mathcal{L}^2 \min\big\{T^{\frac{1}{a+b}}, \exp\{c_3(a,b)(\log M)^{2}\log\log T\}\big\}.
	\end{split}
	\end{equation}
	 Denote $M=\max\{H_1, H_2, R_1, R_2\}$ as before. Then
	\begin{equation}\label{u21}
		U_2\ll \mathcal{L}^3+M^{-\frac{a}{a+b}}\mathcal{L}^2 \min\big\{T^{\frac{1}{a+b}}, \exp\{c_3(a,b)(\log M)^{2}\log\log T\}\big\}.
	\end{equation}

If $T^{\frac{1}{a+b}}\leq \exp\{c_3(a,b)(\log M)^{2}\log\log T\}$, then
		\begin{equation*}
		M\geq \exp\{(a+b)^{-\frac{1}{2}}(c_3(a,b))^{-\frac{1}{2}}(\log T)^\frac{1}{2}(\log\log T)^{-\frac{1}{2}}\}.
    	\end{equation*}
	Thus,
		\begin{equation*}
		U_2\ll T^{\frac{1}{a+b}}\mathcal{L}^2\exp\{-a(a+b)^{-\frac{3}{2}}(c_3(a,b))^{-\frac{1}{2}}(\log T)^\frac{1}{2}(\log\log T)^{-\frac{1}{2}}\}.
	\end{equation*}
	
	If $T^{\frac{1}{a+b}}> \exp\{c_3(a,b)(\log M)^{2}\log\log T\}$, then
	\[ M<\exp\{(a+b)^{-\frac{1}{2}}(c_3(a,b))^{-\frac{1}{2}}(\log T)^\frac{1}{2}(\log\log T)^{-\frac{1}{2}}\}. \]
	The right hand of \eqref{u21}  increases as $M$  increases,  so we still have
	\begin{equation*}
		U_2\ll T^{\frac{1}{a+b}}\mathcal{L}^2\exp\{-a(a+b)^{-\frac{3}{2}}(c_3(a,b))^{-\frac{1}{2}}(\log T)^\frac{1}{2}(\log\log T)^{-\frac{1}{2}}\}.
	\end{equation*}
	
	Combining these two cases for $U_2$ and the upper bound \eqref{u1} for $U_1$, we obtain from \eqref{Uab} that
	\begin{equation*}\label{Uabu}
		U_{a,b}(T; H_1, H_2, R_1,R_2)\ll T^{\frac{1}{a+b}}\mathcal{L}^2\exp\{-a(a+b)^{-\frac{3}{2}}(c_3(a,b))^{-\frac{1}{2}}(\log T)^\frac{1}{2}(\log\log T)^{-\frac{1}{2}}\}.
	\end{equation*}
	By a splitting argument, we deduce that
	\begin{equation*}
	\Sigma_2(T; H,R, a,b)
	\ll T^{\frac{1}{a+b}}\mathcal{L}^6\exp\{-a(a+b)^{-\frac{3}{2}}(c_3(a,b))^{-\frac{1}{2}}(\log T)^\frac{1}{2}(\log\log T)^{-\frac{1}{2}}\}.
	\end{equation*}
	A similar estimate can be obtained for $\Sigma_2(T; H,R, b, a)$. Thus we get part (I) of Lemma \ref{lem:Sab2}.
	
	We now prove part (II) and suppose $a, b\in  \mathbb{Z}^+$ with $b\geq a$ and $(a,b)=1$. The proof follows a similar structure to part (I), with the key difference being the substitution of a new lower bound for $|\eta(h_1,r_1,h_2,r_2)|$, as the estimate differs in this case.
	
	By using the mean value theorem and \eqref{asymphr}, we have
	\[|\eta(h_1,r_1,h_2,r_2)|\asymp\frac{1}{a+b}(h_1^ar_1^b)^{\frac{1}{a+b}-1}|h_1^ar_1^b-h_2^ar_2^b|\gg (H_1^aR_1^b)^{\frac{1}{a+b}-1}. \]
	Using this estimate to replace \eqref{etal}, we see from the condition for $U_1$ that
	\[ T^\frac{1}{a+b}\ll(H_1^aR_1^b)^{1-\frac{1}{a+b}},\]
	which gives
	\[ H_1^aR_1^b\gg T^\frac{1}{a+b-1}.\]
	From \eqref{u11}, we obtain
	\begin{equation*}
		U_1\ll T^{\frac{1}{a+b}-\frac{a}{b(a+b)(a+b-1)}}(H_1H_2)^{-\frac{b-a}{2b}}\mathcal{L}^2.
	\end{equation*}
	
	For $U_2$,  the condition for $\eta(h_1,r_1,h_2,r_2)$  can be written as
	\[ \max\big\{T^{-\frac{1}{a+b}}, (H_1^aR_1^b)^{\frac{1}{a+b}-1}\big\}\ll |\eta(h_1,r_1,h_2,r_2)|<\frac{1}{5}(H_1^aR_1^bH_2^aR_2^b)^\frac{1}{2(a+b)}.\]
	Then, \eqref{u20} can be replaced by
		\begin{equation*}
		U_2=\mathcal{L}+(H_1H_2)^{-\frac{b-a}{2b}}(H_1^aR_1^bH_2^aR_2^b)^{-\frac{a}{2b(a+b)}}\mathcal{L}^2\min\big\{T^{\frac{1}{a+b}}, (H_1^aR_1^b)^{1-\frac{1}{a+b}}\big\},
    	\end{equation*}
	which yields
		\begin{align*}
		\notag	U_2&\ll\mathcal{L}+(H_1H_2)^{-\frac{b-a}{2b}}\mathcal{L}^2\min\big\{T^{\frac{1}{a+b}}(H_1^aR_1^b)^{-\frac{a}{b(a+b)}}, (H_1^aR_1^b)^{1-\frac{1}{b}}\big\}\\
		\notag	&\ll\mathcal{L}+(H_1H_2)^{-\frac{b-a}{2b}}\mathcal{L}^2\big(T^{\frac{1}{a+b}}(H_1^aR_1^b)^{-\frac{a}{b(a+b)}}\big)^{1-\frac{a}{b(a+b-1)}}(H_1^aR_1^b)^{(1-\frac{1}{b})\cdot\frac{a}{b(a+b-1)}}\\
		\notag	&\ll\mathcal{L}+(H_1H_2)^{-\frac{b-a}{2b}}T^{\frac{1}{a+b}-\frac{a}{b(a+b)(a+b-1)}}\mathcal{L}^2.
	\end{align*}
	Thus part (II) follows by a splitting argument, and we complete the proof of Lemma \ref{lem:Sab2}.
\end{proof}

\section{ Proof of  Theorem \ref{thm:meansquare_i}}

In this section, we shall prove Theorem \ref{thm:meansquare_i} and consider the integral $\int_{T}^{T+T_0}\Delta^2_{a,b}(x)dx$, where $T\geq 10$ a large parameter, $0<T_0\leq T$ and $\dfrac{a}{b}\in \mathbb{\overline{Q}}\setminus \mathbb{Q}$.

\subsection{ Mean-square of $\Delta^*_{a,b}(x; H)$}\label{ED*}\ 

We first evaluate the integral $\int_{T}^{T+T_0}\Delta^{*2}_{a,b}(x;H)dx$, which provides the main term in the asymptotic formula of $ \int_{T}^{T+T_0}\Delta^2_{a,b}(x)dx$.

 From the expression of $\Delta^*_{a,b}(x, H)$ given by (\ref{Dlt*}), using the product-to-sum trigonometric formula 
 \begin{equation}\label{cos}
 	\cos a_1\cos a_2=\frac{1}{2}\cos(a_1- a_2)+\frac{1}{2}\cos(a_1+a_2),
 \end{equation}
 we have
 \begin{equation}\label{D*2}
 	\int_{T}^{T+T_0}\Delta^{*2}_{a,b}(x;H)dx=S_1+S_2+S_3,
 \end{equation}
 where
 \begin{align*}
 	S_1&=\frac{c_1^2(a,b)}{2\pi^2}\sum_{\substack{1\leq h_1,h_2\leq aH\\1\leq r_1,r_2\leq bH\\ h_1^ar_1^b=h_2^ar_2^b}} (h_1h_2)^{-\frac{a+2b}{2(a+b)}}(r_1r_2)^{-\frac{2a+b}{2(a+b)}}\int_{T}^{T+T_0}x^\frac{1}{a+b}dx,\\
 	S_2&=\frac{c_1^2(a,b)}{2\pi^2}\sum_{\substack{1\leq h_1,h_2\leq aH\\1\leq r_1,r_2\leq bH\\ h_1^ar_1^b\neq h_2^ar_2^b}} (h_1h_2)^{-\frac{a+2b}{2(a+b)}}(r_1r_2)^{-\frac{2a+b}{2(a+b)}}\\
 	&\qquad\qquad\times \int_{T}^{T+T_0}x^\frac{1}{a+b}\cos\Big(2\pi c_2(a,b)x^\frac{1}{a+b}\big((h_1^ar_1^b)^\frac{1}{a+b}-(h_2^ar_2^b)^\frac{1}{a+b}\big)\Big)dx,\\
 	S_3&=\frac{c_1^2(a,b)}{2\pi^2}\sum_{\substack{1\leq h_1,h_2\leq aH\\1\leq r_1,r_2\leq bH}} (h_1h_2)^{-\frac{a+2b}{2(a+b)}}(r_1r_2)^{-\frac{2a+b}{2(a+b)}}\\
 	&\qquad\qquad\times \int_{T}^{T+T_0}x^\frac{1}{a+b}\sin\Big(2\pi c_2(a,b)x^\frac{1}{a+b}\big((h_1^ar_1^b)^\frac{1}{a+b}+(h_2^ar_2^b)^\frac{1}{a+b}\big)\Big)dx.
 \end{align*}
 
 Using Lemma \ref{lem:S1Hab}, we get
 \begin{equation}\label{is1}
 	S_1(x)=\frac{c_1^2(a,b)}{2\pi^2}G_{a,b}\int_{T}^{T+T_0}x^\frac{1}{a+b}dx+O\big(T^{1+\frac{1}{a+b}}H^{-\frac{a}{2(a+b)}}\mathcal{L}^2\big).
 \end{equation}
 By the first derivative test, we obtain
 \begin{equation}\label{is3}
 	\begin{split}
 			S_3(x)&\ll\sum_{\substack{1\leq h_1,h_2\leq aH\\1\leq r_1,r_2\leq bH}}(h_1h_2)^{-\frac{a+2b}{2(a+b)}}(r_1r_2)^{-\frac{2a+b}{2(a+b)}}\frac{T}{(h_1^ar_1^b)^\frac{1}{a+b}+(h_2^ar_2^b)^\frac{1}{a+b}}\\
 		&\ll T\sum_{\substack{1\leq h_1,h_2\leq aH\\1\leq r_1,r_2\leq bH}}(h_1h_2)^{-\frac{a+2b}{2(a+b)}}(r_1r_2)^{-\frac{2a+b}{2(a+b)}}\frac{1}{(h_1^ar_1^bh_2^ar_2^b)^\frac{1}{2(a+b)}}\\
 		&\ll T\mathcal{L}^4,
 	\end{split}
 \end{equation}
 where we used the inequality $a^2_1+a^2_2\geq 2a_1a_2$ in the second step. 
 
 Now we consider the contribution of $S_2(x)$. Write
 \begin{equation}\label{s2}
 	S_2(x)=S_{21}(x)+S_{22}(x),
 \end{equation}
 with
 \begin{multline*}
 	S_{21}(x)=\frac{c_1^2(a,b)}{2\pi^2}\sum_{\substack{1\leq h_1,h_2\leq aH\\1\leq r_1,r_2\leq bH\\|(h_1^ar_1^b)^\frac{1}{a+b}-(h_2^ar_2^b)^\frac{1}{a+b}|\geq\frac{1}{10} (h_1^ar_1^bh_2^ar_2^b)^\frac{1}{2(a+b)}}} (h_1h_2)^{-\frac{a+2b}{2(a+b)}}(r_1r_2)^{-\frac{2a+b}{2(a+b)}}\\
 	\times \int_{T}^{T+T_0}x^\frac{1}{a+b}\cos\Big(2\pi c_2(a,b)x^\frac{1}{a+b}\big((h_1^ar_1^b)^\frac{1}{a+b}-(h_2^ar_2^b)^\frac{1}{a+b}\big)\Big)dx,
 \end{multline*}
 \begin{multline*}
	S_{22}(x)=\frac{c_1^2(a,b)}{2\pi^2}
\sum_{\substack{1\leq h_1,h_2\leq aH\\1\leq r_1,r_2\leq bH\\0<|(h_1^ar_1^b)^\frac{1}{a+b}-(h_2^ar_2^b)^\frac{1}{a+b}|<\frac{1}{10} (h_1^ar_1^bh_2^ar_2^b)^\frac{1}{2(a+b)}}}(h_1h_2)^{-\frac{a+2b}{2(a+b)}}(r_1r_2)^{-\frac{2a+b}{2(a+b)}}\\
	\times \int_{T}^{T+T_0}x^\frac{1}{a+b}\cos\Big(2\pi c_2(a,b)x^\frac{1}{a+b}\big((h_1^ar_1^b)^\frac{1}{a+b}-(h_2^ar_2^b)^\frac{1}{a+b}\big)\Big)dx.
\end{multline*}
By a similar argument as for $S_3(x)$, with $|(h_1^ar_1^b)^\frac{1}{a+b}-(h_2^ar_2^b)^\frac{1}{a+b}|\geq\frac{1}{10} (h_1^ar_1^bh_2^ar_2^b)^\frac{1}{2(a+b)}$ taking place of $|(h_1^ar_1^b)^\frac{1}{a+b}+(h_2^ar_2^b)^\frac{1}{a+b}|\geq 2 (h_1^ar_1^bh_2^ar_2^b)^\frac{1}{2(a+b)}$ in \eqref{is3}, we obtain that
 \begin{equation*}
	S_{21}(x)\ll T\mathcal{L}^4.
\end{equation*}
For $S_{22}$, by the first derivative test and Lemma \ref{lem:Sab2} (I), we get
\begin{align*}
	\ S_{22}(x)	
	\ll T ~\Sigma_2(T; aH,bH, a,b)
	\ll T^{1+\frac{1}{a+b}}\exp\big\{-c_4(a,b)(\log T)^\frac{1}{2}(\log\log T)^{-\frac{1}{2}}\big\}.
\end{align*}
Substituting into \eqref{s2}, we see
 \begin{equation*}
 	S_{2}(x)\ll T^{1+\frac{1}{a+b}}\exp\big\{-c_4(a,b)(\log T)^\frac{1}{2}(\log\log T)^{-\frac{1}{2}}\big\}.
 \end{equation*}
 
 Combining this with \eqref{D*2}, \eqref{is1} and \eqref{is3}, we arrive at
 \begin{multline}\label{iD*}
 	\int_{T}^{T+T_0}\Delta^{*2}_{a,b}(x;z)dx=\frac{c_1^2(a,b)}{2\pi^2}G_{a,b}\int_{T}^{T+T_0}x^\frac{1}{a+b}dx\\
 	+O\Big(T^{1+\frac{1}{a+b}}H^{-\frac{a}{2(a+b)}}\mathcal{L}^2+T^{1+\frac{1}{a+b}}\exp\big\{-c_4(a,b)(\log T)^\frac{1}{2}(\log\log T)^{-\frac{1}{2}}\big\}\Big).
 \end{multline}

\subsection{ Mean value of $\Delta^*_{a,b}(x, H)(\mathcal{R}_1(a,b; x)+\mathcal{R}_1(b,a; x))$}\label{sec:dr}\ 

From the expressions of $\Delta^*_{a,b}(x, H)$ and $\mathcal{R}_1(a,b; x)$ given by
\eqref{R1ab} and \eqref{Dlt*}, we have
\begin{multline*}
	\Delta^*_{a,b}(x, H)\mathcal{R}_1(a,b; x)=\frac{c^2_1(a,b)}{\pi^2}x^\frac{1}{a+b}\sum_{\substack{1\leq h_1,h_2\leq aH,1\leq r_1\leq bH\\bH<r_2\ll h_2c^{\frac{a+b}{a}J}}}\!\! (h_1h_2)^{-\frac{a+2b}{2(a+b)}}(r_1r_2)^{-\frac{2a+b}{2(a+b)}}\\
	\times\cos\Big(2\pi c_2(a,b)x^\frac{1}{a+b}(h_1^ar_1^b)^\frac{1}{a+b}-\frac{\pi}{4}\Big)	 
	 \cos\Big(2\pi c_2(a,b)x^\frac{1}{a+b}(h_2^ar_2^b)^\frac{1}{a+b}-\frac{\pi}{4}\Big).	 
\end{multline*}
By the product-to-sum trigonometric formula  \eqref{cos} and the first derivative test, we get
\begin{equation*}
	\int_{T}^{T+T_0}\Delta^*_{a,b}(x, H)\mathcal{R}_1(a,b; x)dx\ll S_4+S_5,
\end{equation*}
where
\begin{align*}
	S_4&=T^{1+\frac{1}{a+b}}\sum_{\substack{1\leq h_1,h_2\leq aH,1\leq r_1\leq bH\\bH<r_2\ll h_2c^{\frac{a+b}{a}J}\\ h_1^ar_1^b=h_2^ar_2^b }}\!\! (h_1h_2)^{-\frac{a+2b}{2(a+b)}}(r_1r_2)^{-\frac{2a+b}{2(a+b)}},\\
	S_5&=T\sum_{\substack{1\leq h_1,h_2\leq aH,1\leq r_1\leq bH\\bH<r_2\ll h_2c^{\frac{a+b}{a}J}\\ h_1^ar_1^b\neq h_2^ar_2^b }}\!\!(h_1h_2)^{-\frac{a+2b}{2(a+b)}}(r_1r_2)^{-\frac{2a+b}{2(a+b)}}\min\Big(T^\frac{1}{a+b},\frac{1}{|(h_1^ar_1^b)^\frac{1}{a+b}-(h_2^ar_2^b)^\frac{1}{a+b}|}\Big).
\end{align*}
Using Lemma \ref{lem:S1Hab}, we see
\[ S_4\ll T^{1+\frac{1}{a+b}}H^{-\frac{a}{2(a+b)}}\mathcal{L}^2.\]
For $S_5(x)$,  using a similar approach to that of  $S_2(x)$ in subsection \ref{ED*} and noting that $\log(h_2c^{\frac{a+b}{a}J})\ll \log T$, we obtain
\begin{equation*}
	S_5 \ll T^{1+\frac{1}{a+b}}\exp\big\{-c_4(a,b)(\log T)^\frac{1}{2}(\log\log T)^{-\frac{1}{2}}\big\}.
\end{equation*}

Hence,
\begin{equation*}
	\int_{T}^{T+T_0}\Delta^*_{a,b}(x, H)\mathcal{R}_1(a,b; x)dx\ll T^{1+\frac{1}{a+b}}\Big(H^{-\frac{a}{2(a+b)}}\mathcal{L}^2+\exp\big\{-c_4(a,b)(\log T)^\frac{1}{2}(\log\log T)^{-\frac{1}{2}}\big\}\Big).
\end{equation*}
The same estimate holds for $\int_{T}^{T+T_0}\Delta^*_{a,b}(x, H)\mathcal{R}_1(b,a; x)dx$. Thus,
\begin{multline}\label{idr}
		\int_{T}^{T+T_0}\Delta^*_{a,b}(x, H)\big(\mathcal{R}_1(a,b; x)+\mathcal{R}_1(b,a; x)\big)dx\\
		 \ll T^{1+\frac{1}{a+b}}\Big(H^{-\frac{a}{2(a+b)}}\mathcal{L}^2+\exp\big\{-c_4(a,b)(\log T)^\frac{1}{2}(\log\log T)^{-\frac{1}{2}}\big\}\Big).
\end{multline}

\subsection{ Mean-square of $\mathcal{R}_1(a,b; x)+\mathcal{R}_1(b,a; x)$}\ 

From the expression of $\mathcal{R}_1(a,b; x)$ given by \eqref{R1ab}, we have
\begin{multline*}
	\mathcal{R}_1^2(a,b;x)\ll x^\frac{1}{a+b}\sum_{\substack{1\leq h_1,h_2\leq aH,\\bH<r_1,r_2\ll Hc^{\frac{a+b}{a}J}}}\!\! (h_1h_2)^{-\frac{a+2b}{2(a+b)}}(r_1r_2)^{-\frac{2a+b}{2(a+b)}}\\
	\times\cos\Big(2\pi c_2(a,b)x^\frac{1}{a+b}(h_1^ar_1^b)^\frac{1}{a+b}-\frac{\pi}{4}\Big)	 
	\cos\Big(2\pi c_2(a,b)x^\frac{1}{a+b}(h_2^ar_2^b)^\frac{1}{a+b}-\frac{\pi}{4}\Big).
\end{multline*}
By  \eqref{cos} and the first derivative test, we see
\begin{equation*}
			\int_{T}^{T+T_0}\mathcal{R}_1^2(a,b;x)dx\ll S_6(x)+S_7(x),
\end{equation*}
where
\begin{align*}
	S_6&=T^{1+\frac{1}{a+b}}\sum_{\substack{1\leq h_1,h_2\leq aH\\bH<r_1,r_2\ll Hc^{\frac{a+b}{a}J}\\ h_1^ar_1^b=h_2^ar_2^b }}\!\! (h_1h_2)^{-\frac{a+2b}{2(a+b)}}(r_1r_2)^{-\frac{2a+b}{2(a+b)}},\\
	S_7&=T\sum_{\substack{1\leq h_1,h_2\leq aH\\bH<r_1,r_2\ll Hc^{\frac{a+b}{a}J}\\ h_1^ar_1^b\neq h_2^ar_2^b }}\!\!(h_1h_2)^{-\frac{a+2b}{2(a+b)}}(r_1r_2)^{-\frac{2a+b}{2(a+b)}}\min\Big(T^\frac{1}{a+b},\frac{1}{|(h_1^ar_1^b)^\frac{1}{a+b}-(h_2^ar_2^b)^\frac{1}{a+b}|}\Big).
\end{align*}
Using Lemma \ref{lem:S1Hab}, we get
\begin{equation*}
   S_6\ll T^{1+\frac{1}{a+b}}H^{-\frac{a}{2(a+b)}}\mathcal{L}^2.
\end{equation*}
A similar arguement as for $S_5$ in subsection \ref{sec:dr} gives
\begin{equation*}
	S_7\ll T^{1+\frac{1}{a+b}}\exp\big\{-c_4(a,b)(\log T)^\frac{1}{2}(\log\log T)^{-\frac{1}{2}}\big\}.
\end{equation*}
Hence,
\begin{equation*}
	\int_{T}^{T+T_0}\mathcal{R}_1^2(a,b;x)dx\ll T^{1+\frac{1}{a+b}}\Big(H^{-\frac{a}{2(a+b)}}\mathcal{L}^2+\exp\big\{-c_4(a,b)(\log T)^\frac{1}{2}(\log\log T)^{-\frac{1}{2}}\big\}\Big).
\end{equation*}

Similarly, we can get the same estimate for $\int_{T}^{T+T_0}\mathcal{R}_1^2(b,a;x)dx$. Therefore,
\begin{multline}\label{ir1}
		\int_{T}^{T+T_0}\big(\mathcal{R}_1(a,b; x)+\mathcal{R}_1(b,a; x)\big)^2dx\\
		\ll   T^{1+\frac{1}{a+b}}\Big(H^{-\frac{a}{2(a+b)}}\mathcal{L}^2+\exp\big\{-c_4(a,b)(\log T)^\frac{1}{2}(\log\log T)^{-\frac{1}{2}}\big\}\Big).
\end{multline}

\subsection{ Mean-square of $G(a,b;x)+G(b,a;x)$}\ 

From the expression of $G(a,b;x)$ \eqref{Gab}, we have
\begin{align*}
		\int_{T}^{T+T_0}G(a,b;x)dx&\ll \sum_{m\leq 2T^\frac{1}{a+b}}\int_{T}^{T+T_0}\min\Big(1,\frac{1}{H\Big\|\frac{x^\frac{1}{a}}{m^\frac{b}{a}}\Big\|}\Big)dx\\
		&\ll \sum_{m\leq 2T^\frac{1}{a+b}}\int_{(\frac{T}{m^b})^\frac{1}{a}}^{(\frac{2T}{m^b})^\frac{1}{a}}\min\big(1,\frac{1}{H\|u\|}\big)m^bu^{a-1}du\\
		&\ll \sum_{m\leq 2T^\frac{1}{a+b}}m^b\Big(\frac{T}{m^b}\Big)^\frac{a-1}{a}\Big(\frac{T}{m^b}\Big)^\frac{1}{a}\int_{0}^{1}\min\big(1,\frac{1}{H\|u\|}\big)du\\
		&\ll T\sum_{m\leq 2T^\frac{1}{a+b}}\int_{0}^{\frac{1}{2}}\min\big(1,\frac{1}{H\|u\|}\big)du\\
		&\ll T^{1+\frac{1}{a+b}}\Big(\int_{0}^{\frac{1}{H}}du+\int_{\frac{1}{H}}^{\frac{1}{2}}\frac{1}{Hu}du\Big)\\
		&\ll T^{1+\frac{1}{a+b}}H^{-1}\mathcal{L}.
\end{align*}
Combining with the trivial bound $G(a,b;x)\ll T^{\frac{1}{a+b}}$, we get 
\begin{equation*}
	\int_{T}^{T+T_0}G^2(a,b;x)dx\ll T^{1+\frac{2}{a+b}}H^{-1}\mathcal{L},
\end{equation*}
which holds for $\int_{T}^{T+T_0}G^2(b,a;x)dx$ as well. So we obtain that
 \begin{equation}\label{error1}
 	\int_{T}^{T+T_0}\big(G(a,b;x)+G(b,a;x)\big)^2dx\ll T^{1+\frac{2}{a+b}}H^{-1}\mathcal{L}.
 \end{equation}
	
\subsection{ Mean-square of $\Delta_{a,b}(x)$}\ 

In this subsection, we always take $H=T^{2(a+b)}.$ By \eqref{ir1} and \eqref{error1}, we get
\begin{equation}\label{E2}
	\begin{split}
			&\ \quad\int_{T}^{T+T_0}\big(\mathcal{R}_1(a,b; x)+\mathcal{R}_1(b,a; x)+G(a,b;x)+G(b,a;x)+x^\frac{1}{3(a+b)}H^\varepsilon\big)^2dx\\
		&\ll T^{1+\frac{1}{a+b}}\Big(H^{-\frac{a}{2(a+b)}}\mathcal{L}^2+\exp\big\{-c_4(a,b)(\log T)^\frac{1}{2}(\log\log T)^{-\frac{1}{2}}\big\}+T^{\frac{1}{a+b}}H^{-1}\mathcal{L}+T^{-\frac{1}{3(a+b)}}H^\varepsilon\Big)\\
		&\ll T^{1+\frac{1}{a+b}}\exp\big\{-c_4(a,b)(\log T)^\frac{1}{2}(\log\log T)^{-\frac{1}{2}}\big\}.
	\end{split}
\end{equation}
From \eqref{iD*} and \eqref{error1}, using Cauchy-Schwarz inequality, we have 
\[ \int_{T}^{T+T_0}\Delta^*_{a,b}(x)\big(G(a,b;x)+G(b,a;x)+x^\frac{1}{3(a+b)}H^\varepsilon\big)dx\ll T^{1+\frac{3}{2(a+b)}}H^{-\frac{1}{2}}\mathcal{L}+T^{1+\frac{5}{6(a+b)}}H^\varepsilon.\]
Combining this with \eqref{idr}, we see
\begin{equation}\label{DE}
\begin{split}
	&\quad \int_{T}^{T+T_0}\Delta^*_{a,b}(x)\big(\mathcal{R}_1(a,b; x)+\mathcal{R}_1(b,a; x)+G(a,b;x)+G(b,a;x)+x^\frac{1}{3(a+b)}H^\varepsilon\big)dx\\
	&\ll T^{1+\frac{1}{a+b}}\Big(H^{-\frac{a}{2(a+b)}}\mathcal{L}^2+\exp\big\{-c_4(a,b)(\log T)^\frac{1}{2}(\log\log T)^{-\frac{1}{2}}\big\}+T^{\frac{1}{2(a+b)}}H^{-\frac{1}{2}}\mathcal{L}+T^{-\frac{1}{6(a+b)}}H^\varepsilon\Big)\\
	&\ll T^{1+\frac{1}{a+b}}\exp\big\{-c_4(a,b)(\log T)^\frac{1}{2}(\log\log T)^{-\frac{1}{2}}\big\}.
\end{split}
\end{equation}

 Since
\[ \Delta^{2}_{a,b}(x)=\Delta^{*2}_{a,b}(x;z)+2\Delta^{*2}_{a,b}(x;z)\big(\Delta_{a,b}(x)-\Delta^{*}_{a,b}(x)\big)+\big(\Delta_{a,b}(x)-\Delta^{*}_{a,b}(x)\big)^2,\]
from \eqref{eabr}, \eqref{Dlt}, \eqref{iD*}, \eqref{E2} and \eqref{DE}, by noting that $H=T^{2(a+b)}$, we obtain
\begin{multline*}
	\int_{T}^{T+T_0}\Delta^{2}_{a,b}(x)dx=\frac{c_1^2(a,b)}{2\pi^2}G_{a,b}\int_{T}^{T+T_0}x^\frac{1}{a+b}dx\\
	+O\Big(T^{1+\frac{1}{a+b}}\exp\big\{-c_4(a,b)(\log T)^\frac{1}{2}(\log\log T)^{-\frac{1}{2}}\big\}\Big).
\end{multline*}
Hence, the proof of Theorem \ref{thm:meansquare_i} is completed. 


\section{ Proof of Theorem \ref{thm:meansquare_q} }

In this section,  we consider the case where $1\leq a<b$ are fixed integers such that $(a, b)=1$. 

By applying the same procedure used in the proof of Theorem \ref{thm:meansquare_i}, and substituting Lemma 4.2 (II) for Lemma 4.2 (I), we can derive that
 \begin{multline}\label{iD*q}
	\int_{T}^{T+T_0}\Delta^{*2}_{a,b}(x;z)dx=\frac{c_1^2(a,b)}{2\pi^2}G_{a,b}\int_{T}^{T+T_0}x^\frac{1}{a+b}dx\\
	+O\big(T^{1+\frac{1}{a+b}}H^{-\frac{a}{2(a+b)}}\mathcal{L}^2+T^{1+\frac{1}{a+b}-\frac{a}{b(a+b)(a+b-1)}}\mathcal{L}^4\big),
\end{multline}
\begin{equation}\label{idrq}
	\int_{T}^{T+T_0}\Delta^*_{a,b}(x, H)\big(\mathcal{R}_1(a,b; x)+\mathcal{R}_1(b,a; x)\big)dx\ll T^{1+\frac{1}{a+b}}H^{-\frac{a}{2(a+b)}}\mathcal{L}^2+T^{1+\frac{1}{a+b}-\frac{a}{b(a+b)(a+b-1)}}\mathcal{L}^4,
\end{equation}
\begin{equation}\label{ir1q}
	\int_{T}^{T+T_0}\big(\mathcal{R}_1(a,b; x)+\mathcal{R}_1(b,a; x)\big)^2dx\ll   T^{1+\frac{1}{a+b}}H^{-\frac{a}{2(a+b)}}\mathcal{L}^2+T^{1+\frac{1}{a+b}-\frac{a}{b(a+b)(a+b-1)}}\mathcal{L}^4.
\end{equation}

We still take $H=T^{2(a+b)}$. From \eqref{ir1q} and \eqref{error1}, we see
\begin{equation}\label{E2q}
	\begin{split}
		&\ \quad\int_{T}^{T+T_0}\big(\mathcal{R}_1(a,b; x)+\mathcal{R}_1(b,a; x)+G(a,b;x)+G(b,a;x)+\mathcal{L}^3\big)^2dx\\
		&\ll T^{1+\frac{1}{a+b}}H^{-\frac{a}{2(a+b)}}\mathcal{L}^2+T^{1+\frac{1}{a+b}-\frac{a}{b(a+b)(a+b-1)}}\mathcal{L}^4+T^{1+\frac{2}{a+b}}H^{-1}\mathcal{L}\\
		&\ll T^{1+\frac{1}{a+b}-\frac{a}{b(a+b)(a+b-1)}}\mathcal{L}^4.
	\end{split}
\end{equation}
By \eqref{iD*q}, \eqref{error1} and Cauchy-Schwarz inequality, we get
\[ \int_{T}^{T+T_0}\Delta^*_{a,b}(x)\big(G(a,b;x)+G(b,a;x)+\mathcal{L}^3\big)dx\ll T^{1+\frac{3}{2(a+b)}}H^{-\frac{1}{2}}\mathcal{L}+T^{1+\frac{1}{2(a+b)}}\mathcal{L}^3, \]
which combining with \eqref{idrq} yields
\begin{equation}\label{DEq}
	\begin{split}
		&\quad\  \int_{T}^{T+T_0}\Delta^*_{a,b}(x)\big(\mathcal{R}_1(a,b; x)+\mathcal{R}_1(b,a; x)+G(a,b;x)+G(b,a;x)+\mathcal{L}^3\big)dx\\
		&\ll T^{1+\frac{1}{a+b}}H^{-\frac{a}{2(a+b)}}\mathcal{L}^2+T^{1+\frac{1}{a+b}-\frac{a}{b(a+b)(a+b-1)}}\mathcal{L}^4+ T^{1+\frac{3}{2(a+b)}}H^{-\frac{1}{2}}\mathcal{L}+T^{1+\frac{1}{2(a+b)}}\mathcal{L}^3\\
		&\ll T^{1+\frac{1}{a+b}-\frac{a}{b(a+b)(a+b-1)}}\mathcal{L}^4+T^{1+\frac{1}{2(a+b)}}\mathcal{L}^3\\
		&\ll T^{1+\frac{1}{a+b}-\frac{a}{b(a+b)(a+b-1)}}\mathcal{L}^4,
	\end{split}
\end{equation}
where in the last step we used $\frac{a}{b(a+b-1)}<\frac{1}{2}$ when $b>a\geq 1$ are integers. Using
\[ \Delta^{2}_{a,b}(x)=\Delta^{*2}_{a,b}(x;z)+2\Delta^{*2}_{a,b}(x;z)\big(\Delta_{a,b}(x)-\Delta^{*}_{a,b}(x)\big)+\big(\Delta_{a,b}(x)-\Delta^{*}_{a,b}(x)\big)^2,\]
together with \eqref{eabq}, \eqref{Dlt}, \eqref{iD*q}, \eqref{E2q} and \eqref{DEq}, noting that $H=T^{2(a+b)}$,  we deduce that
\begin{equation*}
	\int_{T}^{T+T_0}\Delta^{2}_{a,b}(x)dx=\frac{c_1^2(a,b)}{2\pi^2}G_{a,b}\int_{T}^{T+T_0}x^\frac{1}{a+b}dx
	+O\big(T^{1+\frac{1}{a+b}-\frac{a}{b(a+b)(a+b-1)}}\mathcal{L}^4\big),
\end{equation*}
which completes the proof of Theorem \ref{thm:meansquare_q}.

\section{Proofs of  Corollary \ref{sign_i} and Corollary \ref{sign_q}}

Now we prove Corollary \ref{sign_i} and  \ref{sign_q} by contradiction. 

Let $T>10$ be a large parameter and $T^ {1-\frac{1}{2 (a+ b)}}\ll T_0\leq T$. Assume that $\Delta_{a,b}(x)$ does not change sign in the interval $[T, T+T_0]$. Then, by \eqref{upbd} and \eqref{meanv}, we see
\begin{equation}\label{Fupp}
	\int_{T}^{T+T_0}\Delta^{2}_{a,b}(x)dx\ll \max_{T \leq x \leq T+T_0}\big|\Delta_{a,b}(x)\big|\cdot \Big|\int_{T}^{T+T_0}\Delta_{a,b}(x)dx\Big|=o\big(T_0 T^{\frac{1}{a+b}}\big)
\end{equation}
holds for any positive real numbers $a\neq b$.

%

 Select $T_0$ as  
 $$T_0=c_5(a,b)\exp\big\{-c_4(a,b)(\log T)^\frac{1}{2}(\log\log T)^{-\frac{1}{2}}\big\}$$
  if $\dfrac{a}{b}\in \mathbb{\overline{Q}}\setminus \mathbb{Q}$; and as $$T_0=c_6(a,b)T^{1-\frac{a}{b(a+b)(a+b-1)}}\mathcal{L}^4$$
  if $1\leq a<b$ are integers with $(a, b)=1$, where $c_5(a,b)$, $c_6(a,b)$ are sufficiently large constants. By Theorem \ref{thm:meansquare_i} and Theorem \ref{thm:meansquare_q}, we obtain, in both cases,
 \begin{equation*}
 	\int_{T}^{T+T_0}\Delta^{2}_{a,b}(x)dx\gg T_0T^{\frac{1}{a+b}}.
 \end{equation*}
This leads to a contradiction with \eqref{Fupp}.  Thus we prove Corollary \ref{sign_i} and \ref{sign_q}.

\bibliography{jia}
\bibliographystyle{abbrv}

\end{document}